\documentclass[12pt,reqno]{amsart}

\usepackage{amssymb,hyperref}
\usepackage{amsmath}
\usepackage{amsfonts}
\usepackage[all]{xy}

\newtheorem{theorem}{Theorem}[section]
\newtheorem{proposition}[theorem]{Proposition}

\newtheorem{corollary}[theorem]{Corollary}

\theoremstyle{definition}
\newtheorem{definition}[theorem]{Definition}
\newtheorem{remark}[theorem]{Remark}
\newtheorem{example}[theorem]{Example}
\numberwithin{equation}{section}

\newcommand{\C}{\mathbb{C}}
\newcommand{\R}{\mathbb{R}}

\newcommand{\Z}{\mathbb{Z}}

\begin{document} 

\title[Moduli spaces   flat bundles on Sasakian manifolds with fixed basic structures ]{Non-abelian Hodge correspondence and moduli spaces of  flat bundles on Sasakian manifolds with fixed basic structures}

\author[H. Kasuya]{Hisashi Kasuya}

\address{Department of Mathematics, Graduate School of Science, Osaka University, Osaka,
Japan}

\email{kasuya@math.sci.osaka-u.ac.jp}

\subjclass[2010]{}


\begin{abstract}
We show that the moduli space of simple flat bundles over a compact Sasakian manifold is a finite disjoint union of  moduli spaces of simple flat bundles with fixed basic structures.
This gives a detailed description of the non-abelian Hodge correspondence on a compact Sasakian manifold at the level of moduli spaces.
As an application, we give an  analogue of Hitchin's properness of maps defined by the coefficients of the characteristic polynomial of Higgs fields.
\end{abstract}

\maketitle

\section{Introduction}
As an odd-dimensional analogue of the non-abelian Hodge correspondence on a compact K\"ahler manifold established by Hitchin \cite{Hit}, Corlette \cite{Cor} and Simpson \cite{Si1, Si2}, Biswas and the author \cite{BK, BK2} prove that on a compact Sasakian manifold there is a canonical $1$ to $1$ correspondence between (semi)simple flat bundles and (poly)stable basic Higgs bundles with vanishing basic Chern classes.
This correspondence is a category equivalence and so it induces a bijection between the sets of isomorphism classes.
We are interested in considering this correspondence at the level of moduli spaces.
We may expected that the moduli space of  simple flat bundles and the moduli space of  stable basic Higgs bundles with vanishing basic Chern classes are homeomorphic.
The purpose of this paper is to describe such expected homeomorphism precisely.

Fix a $(2n+1)$-dimensional Sasakian manifold $M$ as a base space.
Let ${\mathcal M}^{s}_{flat }(SL_r)$ be the moduli space of simple flat complex vector bundles over $M$ of rank $r$ with  trivial  determinant.
$M$ admits a canonical $1$-dimensional foliation defined by the Reeb vector field $\xi$.
Let $E$ be a smooth complex vector bundle of rank $r$.
Each flat connection $D$ on $E$ defines a basic structure $D_{\xi}$ associated with the canonical foliation.
Assume $D$ is simple.
As a subspace of ${\mathcal M}^{s}_{flat }(SL_r)$ containing the simple flat bundle $(E,D)$, we define the moduli space  
${\mathcal M}^{s}_{Bflat }(E, D_{\xi})$ of simple flat bundles with trivial determinant and with fixed basic structure $D_{\xi}$.
We prove the following:
\begin{theorem}\label{IMT}
${\mathcal M}^{s}_{Bflat }(E, D_{\xi})$ is an open and  closed set in ${\mathcal M}^{s}_{flat }(SL_r)$ and 
${\mathcal M}^{s}_{flat }(SL_r)$ is a finite disjoint union \[\bigcup_{(E, D_{\xi})} {\mathcal M}^{s}_{Bflat }(E, D_{\xi})\]
where the union runs over the set of isomorphism classes of basic vector bundles $(E, D_{\xi})$ coming from simple flat complex vector bundles $(E, D)$  of rank $r$ with the trivial  determinants.
\end{theorem}

We have the following corollary.
\begin{corollary}\label{Intco}
The set of isomorphism classes of basic vector  bundles of  rank $r$ over a compact Sasakian manifold $M$  induced by semi-simple flat bundles is finite.
\end{corollary}
In case $r=1$, the statement follows from the Hodge theoretical fact that the first de Rham cohomology of a compact Sasakian manifold is isomorphic to the first basic cohomology.
Thus the higher rank case can be seen as  a non-abelian analogue of this fact.

By Theorem \ref{IMT}, we can study the singularity of ${\mathcal M}^{s}_{flat }(SL_r)$ by the second  basic cohomology instead of the second de Rham cohomology.
We have the following smoothness result.

\begin{corollary}
If $\dim M=3$, then ${\mathcal M}^{s}_{flat }(SL_r)$ is a smooth complex manifold.
\end{corollary}

Define the moduli space ${\mathcal M}^{st}_{BHiggs }(E, D_{\xi})$ of  stable basic  Higgs bundles with fixed basic structure $D_{\xi}$.
The non-abelian Hodge correspondence on a compact Sasakian manifold at the level of moduli spaces can be described by the following:
\begin{theorem}
There is a canonical homeomorphism
${\mathcal M}^{s}_{Bflat }(E, D_{\xi}) \cong {\mathcal M}^{st}_{BHiggs }(E, D_{\xi})$.
\end{theorem}

In this paper we mainly consider analytic structures of moduli spaces of  stable objects (simple flat bundles and stable basic Higgs bundles).
Only the algebraic structure we use is a structure of a quasi-projective variety on  ${\mathcal M}^{s}_{flat }(SL_r)$ identified with the quotient of stable points in the representation variety.
Studying algebraic structures on moduli space of basic Higgs bundles and extending to semi-stable objects will be left for the future.
In this paper we only discuss fundamental properties toward such purpose. 
For connecting analytic structures with  algebraic structures and for  extending stable objects to semi-stable objects, 
certain compactness property given in \cite[Lemma 2.8]{Si2} is an important tool.
But, for our case, we have a difficulty giving  similar arguments of Simpson's proof due to the   basicness (see  Proposition \ref{com} and Remark \ref{remcom}).
Applying Corollary \ref{Intco}, we can say that the limit up to unitary gauge transformation preserves fixed basic structures up to isomorphisms and 
we give an analogue  of Simpson's compactness which is compatible with the basic structures (Corollary \ref{Bcom}).
By this, we can obtain a Sasakian analogue of Hitchin's properness theorem.

\begin{corollary}\label{prop}
If every semi-simple flat bundle  with fixed basic structure $D_{\xi}$ is simple, then the map $\sigma: {\mathcal M}^{st}_{BHiggs }(E, D_{\xi})\to \bigoplus ^{{\rm rank} E} _{i=2}H^{0,0}_{B}(M, Sym^{i} T^{1,0\ast})$  defined by   the coefficients of the characteristic polynomial of Higgs fields  is proper.
\end{corollary}
By this, we give explicit examples of  proper $\sigma: {\mathcal M}^{st}_{BHiggs }(E, D_{\xi})\to \bigoplus ^{{\rm rank} E} _{i=2}H^{0,0}_{B}(M, Sym^{i} T^{1,0\ast})$ in case $\dim M=3$ as an analogue of Hitchin's basic example \cite[Theorem 8.1]{Hit}.

\noindent {\bf Acknowledgement.} The author thanks  Dr. Takashi Ono for  valuable discussions  about  moduli spaces  which inspire  to  start  this project. 

\section{Non-abelian Hodge correspondence on compact Sasakian manifolds}
\subsection{Sasakian manifolds}

Let $M$ be a $(2n+1)$-dimensional real $C^\infty$ orientable manifold. 
A {\em CR-structure} on $M$ is an 
$n$-dimensional complex sub-bundle $T^{1,0}M$ of $TM_{\mathbb C}$ such that
$T^{1,0}M\cap T^{0,1}M
\,=\,\{0\}$ and $T^{1,0}M$ is involutive with respect to the Lie bracket
where $T^{0,1}M$ is the complex conjugation $\overline{T^{1,0}M}$  of $T^{1,0}M$.
 A {\em contact CR-structure} is a CR-structure  $T^{1,0}M$ equipped  with a contact 
$1$-form $\eta$ on $M$ such that $\ker\eta\,=(T^{1,0}M\oplus T^{0,1}M)_{\R}$.
A contact CR-structure  is a {\em 
strongly pseudo-convex CR-manifold} if 
 the Hermitian form $L_{\eta}$ on $\ker\eta$ defined by 
$L_{\eta}(X,Y)=d\eta(X,IY)$ for $ X,Y\in \ker\eta$
is positive definite for every point $x\,\in\, M$ where $I$ is the almost complex structure on  $\ker\eta$ defined by $T^{1,0}M$. 
A {\em Sasakian structure} is a strongly pseudo-convex CR-structure  $(T^{1,0}M, \eta)$ such that
  the Reeb vector field $\xi$ for the 
contact form $\eta$ is  an infinitesimal  automorphism  of the CR structure $T^{1,0}M$. 
The Sasakian metric $g_{\eta}$ associated with a Sasakian structure  $(T^{1,0}M, \eta)$ is 
a Riemannian  metric $g_{\eta}$  on $M$  such that the direct sum $TM= \ker\eta\oplus \langle \xi\rangle $ is orthogonal, the restriction of  $g_{\eta}$   on $\ker\eta$ is  $L_{\eta}$ and $g_{\eta}(\xi,\xi)=1$.

\subsection{Hodge theory}\label{subshod}

Let $(M, T^{1,0}M,\eta)$ be a compact Sasakian manifold
of dimension $2n+1$.
Denote by $\xi$ the Reeb vector field for the 
contact form $\eta$. 
The basic de Rham complex $A^{\ast}_{B}(M)$ is defined by the sub-complex of the de Rham complex $A^{\ast}(M)$ as
$A^{\ast}_{B}(M)=\{\alpha\in A^{\ast}(M)\vert \iota_{\xi}\alpha=\iota_{\xi}d\alpha=0\}$.
Since $(\xi, T^{1,0}M)$ defines a transverse holomorphic foliation, we have a canonical double complex structure $(A^{\ast,\ast}_{B}(M), \partial_{B}, \overline\partial_{B})$ like the usual Dolbeault complex on complex manifolds.

For the usual Hodge star operator $\ast: A^{r}(M)\to A^{2n+1-r}(M)$
associated to the Sasakian metric $g_{\eta}$, 
we define the basic Hodge star operator $\star:A^{r}_{B}(M)_{\mathbb C}
\to  A^{2n-r}_{B}(M)_{\mathbb C}$ to be
$\star\omega\,=\,\ast(\eta\wedge \omega)$ for $\omega\,\in\, A^{r}_{B}(M)_{\mathbb C}$.
Define the Hermitian $L^2$-product
\[A^{r}_{B}(M)_{\mathbb C}\times A^{r}_{B}(M)_{\mathbb C}\,\ni\,
(\alpha,\,\beta)\,\longmapsto\, \int_{M} \eta\wedge\alpha\wedge \star\overline{\beta}
\, \in\, {\mathbb C}\, ,
\]
and we define
 $d^{\star}:A^{r}_{B}(M)\to
A^{r-1}_{B}(M)$,
  $\partial_{B}^{\star}:A^{p,q}_{B}(M)\to A^{p-1,q}_{B}(M)$,
 $\overline\partial_{B}^{\star}:A^{p,q}_{B}(M)\to  A^{p,q-1}_{B}(M)$ and 
 $\Lambda: A^{p,q}_{B}(M)
\to A^{p-1,q-1}_{B}(M)$ 
to be the formal adjoints of $d:A^{r}_{B}(M)\to  A^{r+1}_{B}(M)$,\, $\overline\partial_{B}:A^{p,q}_{B}(M)
\to A^{p+1,q}_{B}(M)$,\, $\partial_{B}:A^{p,q}_{B}(M)
\to A^{p,q+1}_{B}(M)$ and $$L
:A^{p,q}_{B}(M)\ni x\mapsto d\eta\wedge x\in A^{p+1,q+1}_{B}(M)$$ respectively.
We notice that the usual co-differential $d^{\ast}=-\ast d \ast $ is not restricted to  $d^{\star}$ on $A^{r}_{B}(M)$ in general.
But we have $d^{\ast}=d^{\star} $ on the basic $1$-forms $A^{1}_{B}(M)$ (see \cite{KT}).
Thus the inclusion $A^{\ast}_{B}(M) \subset A^{\ast}(M)$ induces an injection $H^1_{B}(M)\hookrightarrow H^1(M,\C)$.
Moreover we know that this injection is actually an isomorphism (see \cite{BG}).

We have the K\"ahler identities
\[[\Lambda, \,\partial_{B}]\,=\,
-\sqrt{-1}\,\overline\partial_{B}^{\star}\ \ \text{ and } \ \
[\Lambda ,\,\overline\partial_{B}]\,=\,\sqrt{-1}\,\partial_{B}^{\star}
\]
which  imply the Hodge structure 
\[H^{r}_{B}(M)\,=\,
\bigoplus_{p+q=r} H^{p,q}_{B}(M), \qquad \overline{H^{p,q}_{B}(M)}=H^{q,p}_{B}(M)
\]
where $H^{p,q}_{B}(M)$ is the basic Dolbeault  cohomology i.e.  the cohomology of the bigraded complex  $(A^{\ast,\ast}_{B}(M),\overline\partial_{B})$ (see \cite{EKA}).

\subsection{Basic bundles}

A structure of basic vector bundle (or   basic structure) on a $C^\infty$ vector bundle $E$ is a 
linear
operator $\nabla_{\xi}\,:\, {\mathcal C}^{\infty}(E)
\,\longrightarrow\,
{\mathcal C}^{\infty}(E)$ such that
\[ \nabla_{\xi}(f s)\,=\,f\nabla_{\xi} s+ \xi(f) s.
\]
A differential form $\omega\,\in\, A^{\ast}(M,E)$ with 
values in $E$ is called basic if 
the equations
\[
\iota_{\xi}\omega=0=\nabla_{\xi} \omega
\]
hold where we extend $\nabla_{\xi}: A^{\ast}(M,E)\to A^{\ast}(M,E)$.
Let $A^{\ast}_{B}(M, E)\,\subset\,A^{\ast}(M,E)$ denote the 
subspace of basic forms in the space $A^{\ast}(M,E)$ of differential forms with values in 
$E$.
A Hermitian metric on $E$ is called {\em basic} if it is $\nabla_{\xi}$-parallel.
Unlike the usual Hermitian metric, a basic vector bundle may not admits a basic Hermitian metric in general.
For a connection operator $\nabla : A^{\ast}(M,E)\to A^{\ast+1}(M,E)$ such that the covariant derivative of $\nabla$ along $\xi$ is $\nabla_{\xi}$, the curvature $R^{\nabla}\in A^{2}(M,\, {\rm End} (E))$ is basic i.e. $R^{\nabla}\in A^{2}_{B}(M, {\rm End} (E))$.
For the Chern forms $c_{i}(E,\nabla)\in A^{2i}(M)$ associated with $\nabla$, we have $c_{i}(E,\nabla)\in A_{B}^{2i}(M)$ and we define the basic Chern classes $c_{i,B}(E)\in H^{2i}_{B}(M)$ by the cohomology classes of $c_{i}(E,\nabla)$ in the basic cohomology.

A structure of basic holomorphic bundle (or basic holomorphic  structure) on a $C^\infty$ vector bundle $E$ is a flat  partial  connection defined  on   $\langle \xi\rangle \oplus T^{0,1}$ as in  \cite{Ra}, i.e. 
linear
differential operator $\nabla^{\prime\prime}\,:\, {\mathcal C}^{\infty}(E)
\,\longrightarrow\,
{\mathcal C}^{\infty}(E\otimes ( \langle \xi\rangle \oplus T^{0,1}M)^{\ast})$ such that
\begin{itemize}
\item for any $X\in \langle \xi\rangle \oplus T^{0,1}M$, and any smooth function $f$ on $M$, the equation
\[ \nabla^{\prime\prime}_{X}(f s)\,=\,f\nabla^{\prime\prime}_{X} s+ X(f) s
\]
holds for all smooth sections $s$ of $E$, and
\item if we extend $\nabla^{\prime\prime}$ to $\nabla^{\prime\prime}\,:\, {\mathcal C}^{\infty}(E\otimes \bigwedge^{k} (\langle \xi\rangle \oplus T^{0,1})^{\ast}) \, \longrightarrow\, {\mathcal C}^{\infty}(E\otimes 
\bigwedge^{k+1} (\langle \xi\rangle \oplus T^{0,1}M)^{\ast})$, then $\nabla^{\prime\prime}\circ \nabla^{\prime\prime}\,=\,0$.
\end{itemize}

A basic holomorphic vector bundle $E$ has a canonical basic bundle structure corresponding to the covariant derivative $\nabla^{\prime\prime}_{\xi}$.
$\nabla^{\prime\prime}$ defines the linear operator $\bar\partial_{E}: A^{p,q}_{B}(M,E)\to 
A^{p,q+1}_{B}(M, E)$ so that 
$\bar \partial_{E}( f\omega )=\bar\partial _{B}f \wedge \omega +f \bar\partial_{E}( \omega )$ for $f\in A^{0}_{B}(M), \omega\in A^{p,q}_{B}(M, E)$.
  $(A^{\ast,\ast}_{B}(M,E),\overline\partial_{E})$ is a bigraded complex and we  define the Dolbeault cohomology $H^{\ast,\ast}_{B}(M,E)$ of a basic holomorphic vector bundle $E$ by the cohomology of this complex.
If a basic holomorphic vector bundle $E$ admits a basic Hermitian metric $h$, as complex case, we have a unique unitary connection $\nabla^{h}$ such that for any $X\in \langle \xi\rangle \oplus T^{0,1}M$, $\nabla^{h}_{X}=\nabla^{\prime\prime}_{X}$.
We notice that the bundle $T^{1,0}M$ admits a canonical basic holomorphic structure  such that the unitary  connection associated with the Hermitian metric $L_{\eta}$ is the   Tanaka-Webster connection \cite{Tan, Web} (see \cite{KM}).

For  a  flat vector bundle $(E,D)$ over $M$, we consider a basic structure $D_{\xi}$.
$A^{\ast}_{B}(M,E)$ is a sub-complex of the de Rham 
complex $A^{\ast}(M, E)$ equipped with the differential $D$.
Denote by $H^{\ast}(M,E) $ and $H^{\ast}_{B}(M,E) $ the cohomology of the de Rham complex $A^{\ast}(M,\,E)$ and the
cohomology of the sub-complex $A^{\ast}_{B}(M, E)$ respectively. 
Let $h$ be a (not necessarily basic) Hermitian metric on $E$.
It gives a unique decomposition
$D= \nabla+\phi$
such that $\nabla$ is a unitary connection for $h$ and $\phi$ is a $1$-form on $M$ with values in
the self-adjoint endomorphisms of $E$ with respect to $h$. 
The Hermitian metric
$h$ is called {\em harmonic} if $\nabla^{\ast}\phi=0$ where  
$\nabla^{\ast}$ is  the formal adjoint operator of $\nabla$ with respect to $g$ and $h$. 
Corlette \cite{Cor} proves that  if the flat bundle 
$(E,D)$ is simple, then   there exists a harmonic  metric $h$ which is  unique up to a positive constant.

Assume that the Hermitian metric $h$ is basic.
This condition is equivalent to the condition that $\phi(\xi)\,=\,0$ (see \cite[Proposition 4.1]{BK}).
Then  $\nabla$ is defined on  $ A^{\ast}_{B}(M,E)$.
Decompose the connection $\nabla$ as $\nabla =\partial_{E} +\overline\partial_{E}$
such that  $\partial_{E}:\,A^{p,q}_{B}(M,E)\to
A^{p+1,q}_{B}(M,E)$ and $\overline\partial_{E}:\,A^{p,q}_{B}(M,E)\to A^{p,q+1}_{B}(M,E)$, and also
decompose $\phi$ as 
$\phi\,=\,\theta+\overline\theta$  such that 
$\theta\,\in\, A^{1,0}_{B}(M,\,{\rm End}(E))$ and $ 
\overline\theta\in A^{0,1}_{B}(M,\,{\rm End}(E))$.

\begin{theorem}[{\cite[Theorem 4.2]{BK}}]\label{t4.2}
Let $(M, T^{1,0}M,\eta)$ be a compact Sasakian manifold and
$(E, D)$ a flat complex vector bundle over $M$ with a Hermitian metric $h$.
Then the following two conditions are equivalent:
\begin{itemize}
\item The Hermitian structure $h$ is harmonic, i.e. $(\nabla)^{\ast}\phi\,=\,0$.

\item The Hermitian structure $h$ is basic, and the equations 
\[\overline\partial_{E} \overline\partial_{E}=\,0,\qquad 
\theta\wedge\theta=\,0\qquad {\rm and} \qquad 
\overline\partial_{{\rm End}(E)} \theta\,=\,0\, \]
hold.
\end{itemize}
\end{theorem}

For a harmonic metric $h$, since $h$ is basic, we define the
basic Hodge star operator $\star:A^{r}_{B}(M, E)
\to A^{2n-r}_{B}(M,E^\ast)$, the Hermitian $L^2$-product on $A^{\ast}_{B}(M, E)$ and the formal adjoint operator
$D^{\star}$ of $D$
 in the same manner as the trivial coefficient.
 We have $D^{\ast}=D^{\star}$ on the basic $1$-forms $A^{1}_{B}(M,E)$ as  $d^{\ast}=d^{\star}$ on $A^{1}_{B}(M)$.
We define $D^{\prime}=\partial_{E}+\overline\theta$, $D^{\prime\prime}= \overline\partial_{E}+\theta$ and $D^{c}=\sqrt{-1}(D^{\prime\prime}-D^{\prime})$.

A \textit{basic Higgs bundle} over $M$ is a pair
$(E, \theta)$ consisting of a basic holomorphic vector bundle $E$
and 
$\theta\,\in\, A^{1,0}_{B}(M,\,{\rm 
End}(E))$ satisfying the following two conditions:
$$\overline\partial_{{\rm 
End}(E)}\theta \,=\,0\ \ \text{ and }
\ \ \theta\wedge \theta\,=\,0\, .$$

We define the {\it degree} of a basic holomorphic vector bundle $E$ by 
$${\rm deg}(E)\,:=\,\int_{M}c_{1, B}(E)\wedge (d\eta)^{n-1}\wedge\eta\, .$$
Denote by ${\mathcal O}_{B}$ the sheaf of basic holomorphic functions on $M$, and for a holomorphic vector bundle $E$ on $M$, denote by 
${\mathcal O}_{B}(E)$ the sheaf of basic  holomorphic sections of $E$. 
Consider ${\mathcal 
O}_{B}(E)$ as a coherent ${\mathcal O}_{B}$-sheaf.

For a basic Higgs 
bundle $(E, \theta)$, a {\em sub-Higgs sheaf} of $(E,\theta)$ is a coherent 
${\mathcal O}_{B}$-subsheaf $\mathcal V$ of ${\mathcal O}_{B}(E)$ such that $\theta ({\mathcal V})\, \subset\, {\mathcal 
V}\otimes \Omega_{B}$, where $\Omega_{B}$ is the sheaf 
of basic holomorphic $1$-forms on $M$. By \cite[Proposition 3.21]{BH}, if ${\rm rk} 
(\mathcal V)\,<\,{\rm rk}(E)$ and ${\mathcal O}_{B}(E)/\mathcal V$ is 
torsion-free, then there is a transversely analytic sub-variety $S\,\subset\, M$ of complex 
co-dimension at least 2 such that ${\mathcal V}\big\vert_{M\setminus S}$ is given by a basic holomorphic sub-bundle $V\,\subset\,
E\big\vert_{M\setminus S}$. The degree ${\rm deg}(\mathcal V)$ can be defined by integrating $c_{1, B}(V)\wedge (d\eta)^{n-1}\wedge\eta$
on this complement $M\setminus S$.

\begin{definition}
We say that a basic Higgs bundle $(E, \theta)$ is {\em stable} if $E$ admits a basic Hermitian metric and for every 
sub-Higgs sheaf ${\mathcal V}$ of $(E, \theta)$ such that ${\rm rk} (\mathcal V)\,<\,{\rm 
rk}(E)$ and ${\mathcal O}_{B}(E)/\mathcal V$ is torsion-free, 
the inequality
\[\frac{{\rm deg}(\mathcal V)}{{\rm rk} (\mathcal V)}\,<\,\frac{{\rm deg}(E)}{{\rm rk}(E)}
\]
holds.
\end{definition}
A basic  Higgs bundle $(E, \theta)$ coming from a simple flat bundle equipped with a harmonic metric is stable.
The converse is also true.

Let $(E, \theta)$ be a basic Higgs bundle over a compact Sasakian manifold $M$. 
For  a basic Hermitian metric $h$,
 define $\overline\theta_{h}\,\in\, A^{0,1}_{B}(M,\,{\rm End}(E))$ by
$
h(\theta (e_{1}),\, e_{2})\,=\,h(e_{1}, \,\overline\theta_{h} (e_{2}))
$
for all $e_{1},\,e_{2}\,\in\, E_x$ and all $x\, \in\, M$.
Define the canonical connection
$
D^{h}\,=\,\nabla^{h}+\theta+\overline\theta_{h}
$
on $E$.

\begin{theorem}[\cite{BK, BK2}]\label{harmonicmetric}
If a basic Higgs bundle $(E, \,\theta)$ is stable and satisfies
$$c_{1,B}(E)\,=\,0\qquad {\rm and} \qquad \int_{M} c_{2,B}(E)\wedge(d\eta)^{n-2}\wedge \eta\,=\,0,$$
then there exists a basic Hermitian metric $h$ so that the canonical connection $D^{h}$ is a simple  flat connection 
and such Hermitian metric is unique up to a positive constant.

\end{theorem}

By Theorem \ref{t4.2} a metric $h$ in this theorem  is a Harmonic metric on the simple flat bundle $(E,D^{h})$ with respect to the Sasakian metric $g_{\eta}$.

\subsection{Non-abelian Hodge correspondence}
By the arguments in the last subsection, we have the canonical $1$ to $1$ correspondence between simple flat bundles $(E,D)$ and stable basic Higgs bundles satisfying $c_{1,B}(E)\,=\,0$   and $ \int_{M} c_{2,B}(E)\wedge(d\eta)^{n-2}\wedge \eta\,=\,0$ via harmonic metrics with respect to the Sasakian metric $g_{\eta}$.
We know that this correspondence is actually a category equivalence  \cite[Theorem 7.2]{BK}. 

\section{Moduli spaces}
Let $(M, T^{1,0}M,\eta)$ be a $(2n+1)$-dimensional  compact  Sasakian manifold.
 
\subsection{Moduli spaces of flat connections}

Let $(E,D)$ be a flat complex vector bundle over $M$.
The space of connections on a ${\mathcal C}^{\infty}$ complex vector bundle $E$ is considered as an affine space modeled on $A^{1}(M, {\rm End}(E))$ via $\omega\mapsto D_{\omega}=D+\omega$ with the gauge group $GL(E)$.
Consider the space of flat connections as 
\[\mathcal A_{flat}(E)=\{\omega\in A^{1}(M, {\rm End}(E))\vert D\omega+\omega\wedge \omega=0\}.\]

We assume that $(E, D)$ is simple.
Then $(E,D)$ admits harmonic metric $h$ which is  unique up to positive scaler multiplication by Corlette's Theorem \cite{Cor}.
Denote by ${\rm End}_{0}(E)$ the trace free part of $ {\rm End}(E)$ associated with $h$.
We assume that the determinant of  $(E,D)$ is the  trivial flat line bundle.
Consider 
\[\mathcal A^{s}_{flat}(E)=\{\omega\in \mathcal A_{flat}(E)\cap  A^{1}(M, {\rm End}_{0}(E))\vert (E,D_{\omega}) \,\,{\rm simple} \}\]
with the action  of the  gauge group $SL(E)$ of gauge transformations with determinant $1$.
Define the moduli space  \[{\mathcal M}^{s}_{flat }(E)=\mathcal A^{s}_{flat}(E)/ (SL(E)/ \langle \zeta_{r} I\rangle)\]
where $\zeta_{r}$ is an $r$-th root of unity.
We notice that $SL(E)$ does not acts effectively on $\mathcal A^{s}_{flat}(E)$ and the reduced group    $(SL(E)/ \langle \zeta_{r} I\rangle)$ acts effectively on  $\mathcal A^{s}_{flat}(E)$.

By  the standard arguments (see  \cite[Chapter 7.]{Ko},  \cite[Chapter 4]{LT} and  \cite[Chapter IV]{FM}), ${\mathcal M}^{s}_{flat }(E)$  is a complex analytic space so that 
the local analytic  structure  of  ${\mathcal M}^{s}_{flat }(E)$ is given by the Kuranishi space of  the DGLA $(A^{\ast}(M, {\rm End}_{0}(E)), D)$.
That is, at a point $D$,  ${\mathcal M}^{s}_{flat }(E)$ is locally isomorphic to  a small neighborhood of
 \[Kur(A^{\ast}(M, {\rm End}_{0}(E)))=\{\omega\in A_{flat}(E)\cap  A^{1}(M, {\rm End}_{0}(E))\vert D^{\ast}\omega=0\}\]
at $0$.

\subsection{Moduli spaces of flat connections with fixed basic structures}

Let $(E, D)$ be a flat complex vector bundle over $M$.
Then the space of connections on a ${\mathcal C}^{\infty}$ complex vector bundle $E$ with the  fixed basic structure $D_{\xi}$ is considered as an affine space modeled on $A^{1}_{B}(M, {\rm End}(E))$ with the gauge group $GL_{B}(E)=GL(E)\cap A^{0}_{B}(M, {\rm End}(E))$.
Define
\[ \mathcal A_{Bflat}(E, D_{\xi})=A_{flat}(E)\cap A^{1}_{B}(M, {\rm End}(E)).
\]

We assume that $(E, D)$ is simple.
A  harmonic metric $h$ is basic by Theorem \ref{t4.2}.
Consider 
\[\mathcal A^{s}_{Bflat}(E, D_{\xi})= \mathcal A^{s}_{flat}(E)\cap  A^{1}_{B}(M, {\rm End}_{0}(E))\]
 with the gauge group  $SL_{B}(E)=SL(E)\cap GL_{B}(E)$.
Define the moduli space  
\[{\mathcal M}^{s}_{Bflat }(E, D_{\xi})=\mathcal A^{s}_{Bflat}(E, D_{\xi})/ (SL_{B}(E)/\langle \zeta_{r} I\rangle).\]

By the same manner as ${\mathcal M}^{s}_{flat}(E)$, ${\mathcal M}^{s}_{Bflat }(E, D_{\xi})$ is a complex analytic space so that 
the local analytic structure  of  ${\mathcal M}^{s}_{Bflat }(E, D_{\xi})$ is given by the Kuranishi space of  the DGLA $(A^{\ast}_{B}(M, {\rm End}_{0}(E)), D)$.
That is, at a point $D$,  ${\mathcal M}^{s}_{Bflat}(E, D_{\xi})$ is locally isomorphic to  a small neighborhood of
 \[Kur(A^{\ast}_{B}(M, {\rm End}_{0}(E)))=\{\omega\in \mathcal{A}_{flat}(E)\cap  A^{1}_{B}(M, {\rm End}_{0}(E))\vert D^{\star}\omega=0\}\]
at $0$.
\begin{remark}
The Sobolev space $L^{p}_{k}(A^{\ast}_{B}(M, {\rm End}_{0}(E)))$ for   $A^{\ast}_{B}(M, {\rm End}_{0}(E))$ is the  closure of $A^{\ast}_{B}(M, {\rm End}_{0}(E))$ in  the usual 
Sobolev space $L^{p}_{k}(A^{\ast}(M, {\rm End}_{0}(E)))$ associated with $A^{\ast}(M, {\rm End}_{0}(E))$.
This is isomorphic to the Sobolev completion of $A^{\ast}_{B}(M, {\rm End}_{0}(E))$ associated with the Hermitian $L^2$-product on
$A^{\ast}_{B}(M, {\rm End}_{0}(E))$ defined  in the last section.
We omit the  standard arguments on the  Sobolev space for   $A^{\ast}_{B}(M, {\rm End}_{0}(E))$ which are almost same as the usual Sobolev space.
Only the little  different thing is in  the  boot strapping process.
For proving higher regularity in boot strapping,  we need to extend basic elliptic operators to elliptic operators (see \cite[Lemma 2.8]{BH}) on $A^{\ast}(M, {\rm End}_{0}(E))$ and use the smoothness in $L^{p}_{k}(A^{\ast}(M, {\rm End}_{0}(E)))$ and the relation $L^{p}_{k}(A^{\ast}_{B}(M, {\rm End}_{0}(E)))\cap A^{\ast}(M, {\rm End}_{0}(E))=A^{\ast}_{B}(M, {\rm End}_{0}(E))$  (see \cite{KT}).

\end{remark}

The following is the main statement in this paper.

\begin{theorem}
The natural map ${\mathcal M}^{s}_{Bflat }(E, D_{\xi}) \to  {\mathcal M}^{s}_{flat }(E)$ induced by the inclusion $\mathcal A^{s}_{Bflat}(E, D_{\xi})\subset \mathcal A^{s}_{flat}(E)$  is an open embedding.
\end{theorem}

\begin{proof}
Regard $\mathcal A^{s}_{Bflat}(E, D_{\xi})$ as a space of connections.
Assume for  $D+\omega \in \mathcal A^{s}_{Bflat}(E, D_{\xi})$, $g\in GL(E)$, $g^{-1}(D+\omega)g\in \mathcal A^{s}_{Bflat}(E, D_{\xi})$.
Then $D_{\xi}=g^{-1}(D+\omega  )g(\xi)=g^{-1}D_{\xi}g$.
Thus, $g\in GL_{B}(E)$ and so  the injectivity holds.

By \cite{Ka}, we have the sequences of quasi-isomorphisms of DGLAs
\[
A^{\ast}_{B}(M, {\rm End}_{0}(E))\leftarrow {\rm ker}D^{c}  \to H^{\ast}_{B}(M,{\rm End}_{0}(E))
\]
and 
\[
A^{\ast}(M, {\rm End}_{0}(E))\leftarrow {\rm ker}D^{c} \oplus {\rm ker}D^{c}\wedge \eta \to H^{\ast}_{B}(M,{\rm End}_{0}(E))\oplus H^{\ast}_{B}(M, {\rm End}_{0}(E))\otimes \langle \eta \rangle
\]
where $H^{\ast}_{B}(M,{\rm End}_{0}(E))\oplus H^{\ast}_{B}(M, {\rm End}_{0}(E))\otimes \langle \eta \rangle$ is a DGLA whose differential is defined by the $0$-map on the component $H^{\ast}_{B}(M,{\rm End}_{0}(E))$ and the map
\[H^{\ast}_{B}(M, {\rm End}_{0}(E))\otimes \langle \eta \rangle\ni  \alpha\otimes \eta\mapsto  [d\eta]\wedge \alpha\in H^{\ast}_{B}(M, {\rm End}_{0}(E))
\]
on the  component $H^{\ast}_{B}(M, {\rm End}_{0}(E))\otimes \langle \eta \rangle$.
Thus, the cohomology  map $H^{\ast}_{B}(M, {\rm End}_{0}(E))\to  H^{\ast}(M, {\rm End}_{0}(E))$ induced by 
the inclusion $A^{\ast}_{B}(M, {\rm End}_{0}(E)) \subset  A^{\ast}(M, {\rm End}_{0}(E))$
is identified with the cohomology  map induced by \[H^{\ast}_{B}(M, {\rm End}_{0}(E))\to H^{\ast}_{B}(M,{\rm End}_{0}(E))\oplus H^{\ast}_{B}(M, {\rm End}_{0}(E))\otimes \langle \eta \rangle.\]
By $H^{0}_{B}(M, {\rm End}_{0}(E))=0$, we can say that the inclusion $A^{\ast}_{B}(M, {\rm End}_{0}(E)) \to A^{\ast}(M, {\rm End}_{0}(E))$
 induces an isomorphism on the first cohomology and an injection on the second cohomology.
 Since   $D^{\ast}$ on $A^{1}(M, {\rm End}_{0}(E))$ is restricted to $D^{\star}$ on  $A^{\ast}_{B}(M, {\rm End}_{0}(E))$,  
we have the embedding of the  analytic space
 \[Kur(A^{\ast}_{B}(M, {\rm End}_{0}(E)))=\{\omega\in A_{flat}(E)\cap  A^{1}_{B}(M, {\rm End}_{0}(E))\vert D^{\star}\omega=0\}\] 
 into 
\[Kur(A^{\ast}(M, {\rm End}_{0}(E)))=\{\omega\in A_{flat}(E)\cap  A^{1}(M, {\rm End}_{0}(E))\vert D^{\ast}\omega=0\}.\]
This embedding is  actually  a local  isomorphism 
\[(Kur (A^{\ast}_{B}(M, {\rm End}_{0}(E))),0)\cong (Kur (A^{\ast}(M, {\rm End}_{0}(E))), 0)\] by the Goldman-Millson theorem \cite{GMi2}.
This means that  the natural map ${\mathcal M}^{s}_{Bflat }(E, D_{\xi}) \to  {\mathcal M}^{s}_{flat }(E)$ is a local isomorphism for complex analytic spaces.
\end{proof}

Via the canonical embedding ${\mathcal M}^{s}_{Bflat }(E, D_{\xi}) \hookrightarrow  {\mathcal M}^{s}_{flat }(E)$ in this theorem, 
we have \[ {\mathcal M}^{s}_{flat }(E)=\bigcup_{D_{\xi} } {\mathcal M}^{s}_{Bflat }(E, D_{\xi})\]
 such that the union runs over all basic  structures $ D_{\xi} $ coming from simple  flat connections $D$ on $E$.
Suppose ${\mathcal M}^{s}_{Bflat }(E, D_{\xi})\cap {\mathcal M}^{s}_{Bflat }(E, \tilde{D}_{\xi})\not=\o$.
Then we have $g^{-1}\tilde{D}g=D+\omega$ for some $g\in SL(E)$ and $\omega\in  \mathcal A^{s}_{Bflat}(E, D_{\xi})$.
This implies $g^{-1}\tilde{D}_{\xi}g=D_{\xi}$ and so  for any $\tilde{D}+ \alpha\in \mathcal A^{s}_{Bflat}(E, \tilde{D}_{\xi})$, $g^{-1}(\tilde{D}+\alpha)g=D+\omega +g^{-1}\alpha g\in  \mathcal A^{s}_{Bflat}(E, D_{\xi})$.
Thus, we can say that if ${\mathcal M}^{s}_{Bflat }(E, D_{\xi})\cap {\mathcal M}^{s}_{Bflat }(E, \tilde{D}_{\xi})\not=\o$ then ${\mathcal M}^{s}_{Bflat }(E, D_{\xi})={\mathcal M}^{s}_{Bflat }(E, \tilde{D}_{\xi})$ in ${\mathcal M}^{s}_{flat }(E)$.

We say that basic vector bundles are isomorphic if there exists an isomorphism of smooth bundles commuting with basic structures.
It follows from the above arguments  that the union $\bigcup_{D_{\xi} } {\mathcal M}^{s}_{Bflat }(E, D_{\xi})$ can be seen as   a  disjoint union running over  isomorphism classes of basic vector bundles.
Since the compliment of one stratum ${\mathcal M}^{s}_{Bflat }(E, D_{\xi})$ in ${\mathcal M}^{s}_{flat }(E)$ is a union of other strata, we have the following statement.

\begin{corollary}
${\mathcal M}^{s}_{Bflat }(E, D_{\xi})$ is an open and closed set  in $ {\mathcal M}^{s}_{flat }(E)$.
\end{corollary}

Let ${\mathcal M}^{s}_{flat }(SL_r)$ be the moduli space of simple flat complex vector bundles over $M$ of rank $r$ with the trivial  determinants.
As \cite[Section (8.2)]{Fu}, by \cite[Theorem 1.1]{JM}, ${\mathcal M}^{s}_{flat }(SL_r)$ is analytically   isomorphic to  the quotient $R^{s}(\pi_{1}M, SL_{r}(\C))/SL_{r}(\C)$ of the space of  stable points of the variety of representations.
By geometric invariant theory $R^{s}(\pi_{1}M, SL_{r}(\C))/SL_{r}(\C)$ is a quasi-projective variety.
Thus, ${\mathcal M}^{s}_{flat }(SL_r)$ is a Hausdorff space which     has finitely many  connected components.

\begin{corollary}\label{fidi}
${\mathcal M}^{s}_{flat }(SL_r)$ is a finite disjoint union $\bigcup_{(E, D_{\xi})} {\mathcal M}^{s}_{Bflat }(E, D_{\xi})$ of open and  closed sets such that the union runs over the set of isomorphism classes of basic vector bundles $(E, D_{\xi})$ coming from simple flat complex vector bundles $(E, D)$  of rank $r$ with the trivial  determinants.
\end{corollary}

We have the following important consequence of the finiteness of the union \[{\mathcal M}^{s}_{flat }(SL_r)=\bigcup_{(E, D_{\xi})} {\mathcal M}^{s}_{Bflat }(E, D_{\xi}).\]

\begin{corollary}\label{finiss}
The set of isomorphism classes of basic vector  bundles of  rank $r$ over a compact Sasakian manifold $M$  induced by semi-simple flat bundles is finite.
\end{corollary}
\begin{proof}
In case $r=1$, the statement is consequence of isomorphism from the fact that the first de Rham cohomology $H^{1}(M, \C)$ is isomorphic to the basic cohomology $H^1_{B}(M)$ (see \cite{BoG}).
The case $r\ge 2$ easily  follows from the case $r=1$ and  Corollary \ref{fidi} inductively.
\end{proof}

Since the singularity of ${\mathcal M}^{s}_{flat }(SL_r)$ at a point $(E, D)$ is the singularity of  \[Kur(A^{\ast}_{B}(M, {\rm End}_{0}(E)))=\{\omega\in A_{flat}(E)\cap  A^{1}_{B}(M, {\rm End}_{0}(E))\vert D^{\star}\omega=0\}\] 
at the origin $0$,
we have the following criterion.

\begin{corollary}
${\mathcal M}^{s}_{flat }(SL_r)$ is smooth at a point $(E, D)$ if $H^{2}_{B}(M, {\rm End}_{0}(E))=0$.

\end{corollary}
In case $n=1$ we have an isomorphism $H^{2}_{B}(M, {\rm End}_{0}(E))\cong H^{0}_{B}(M, {\rm End}_{0}(E))=0$ via $\star$.

\begin{corollary}
If $n=1$, then ${\mathcal M}^{s}_{flat }(SL_r)$ is a smooth complex manifold.
\end{corollary}
\begin{remark}
We can say that ${\mathcal M}^{s}_{flat }(SL_r)$ is smooth at a point $(E, D)$ if  $H^{2}(M, {\rm End}_{0}(E))=0$.
But, in case $n=1$, by the poincare duality,  $H^{2}(M, {\rm End}_{0}(E))\not=0$ if and only if $H^{1}(M, {\rm End}_{0}(E))\not=0$.
Thus,  in case $n=1$, $H^{2}(M, {\rm End}_{0}(E))=0$ may not hold in general.

\end{remark}

\subsection{Moduli spaces of harmonic bundles and basic (pluri)harmonic bundles}

Let $(E, D)$ be a flat complex vector bundle over $M$.
Take a Hermitian metric $h$ on $E$.
Then we have the decomposition ${\rm End}_{0}(E)={\frak u}_{0}(E)\oplus P_{0}(E)$ such that  ${\frak u}_{0}(E)$ (resp. $P_{0}(E)$ consists of anti-self-adjoint (resp. self-adjoint) operators.
Decompose $A^{\ast}(M, {\rm End}_{0}(E))=A^{\ast}({\frak u}_{0}(E))\oplus A^{\ast}(P_{0}(E))$.
We assume $(E, D)$ is simple and $h$ is harmonic.
Then for the decomposition $D=\nabla + \phi$ with respect to  $ A^{\ast}({\frak u}_{0}(E))\oplus A^{\ast}(P_{0}(E))$, we have 
$\nabla^{\ast}\phi=0$.

Assume that the determinant of $(E, D)$ is the trivial flat line bundle.
The space of simple harmonic bundle structures  on $E$ is $\mathcal A^{s}_{harm}(E)=\{(\omega \in A^{s}_{flat}(E)\vert  \nabla_{\omega} ^{\ast}\phi_{\omega}=0\}$ with the gauge group  $SU(E)$ of unitary gauge transformations with determinant $1$.
Define the moduli space  \[\mathcal M^{s}_{harm}(E)=\mathcal A^{s}_{harm}(E)/(SU(E)/\langle \zeta_{r} I\rangle).\]
By the existence and uniqueness of harmonic metrics on a simple flat bundle proved by Corlette \cite{Cor},
the inclusion  $\mathcal A^{s}_{harm}(E)\subset \mathcal A^{s}_{flat}(E)$ induces a homeomorphism $\mathcal M^{s}_{harm}(E)\cong {\mathcal M}^{s}_{flat }(E)$.

Since  $h$ is basic, we have  $A^{\ast}_{B}(M, {\rm End}_{0}(E))=A^{\ast}_{B}({\frak u}_{0}(E))\oplus A_{B}^{\ast}(P_{0}(E))$.
For $\omega \in A^{s}_{Bflat}(E, D_{\xi})$, we have the decomposition  $D_{\omega}=\nabla_{\omega} + \phi_{\omega}$ with respect to $A^{\ast}_{B}({\frak u}_{0}(E))\oplus A_{B}^{\ast}(P_{0}(E))$.
Decompose the connection $\nabla_{\omega}$ as $\nabla_{\omega} =\partial_{\omega} +\overline\partial_{\omega}$
where $\partial_{\omega}:\,A^{p,q}_{B}(M,E)\longrightarrow
A^{p+1,q}_{B}(M,E)$ and $\overline\partial_{\omega}:\,A^{p,q}_{B}(M,E)\,\longrightarrow\, A^{p,q+1}_{B}(M,E)$, and also
decompose $\phi_{\omega}$ as 
$\phi_{\omega}=\theta_{\omega}+\overline\theta_{\omega}$ where 
$\theta_{\omega}\in\, A^{1,0}_{B}(M,\,{\rm End}_{0}(E))$ and $ 
\overline\theta_{\omega}\in A^{0,1}_{B}(M,\,{\rm End}_{0}(E))$.
By Theorem \ref{t4.2},   the space of simple harmonic bundle structures  on $E$ with fixed basic structure $D_{\xi}$  is \[\mathcal A^{s}_{Bharm}(E)=\{\omega \in \mathcal A^{s}_{Bflat}(E, D_{\xi})\vert \overline{\partial}_{\omega}\overline{\partial}_{\omega}= \overline{\partial}_{\omega} \theta_{\omega}=\theta_{\omega}\wedge \theta_{\omega}=0\}\]
with the gauge group $SU_{B}(E)=SU(E)\cap GL_{B}(E)$.
 Define the moduli space  \[\mathcal M^{s}_{Bharm}(E, D_{\xi})=\mathcal A^{s}_{Bharm}(E, D_{\xi})/(SU_{B}(E)/\langle \zeta_{r} I\rangle).\]
 By the existence,  uniqueness and basicness  of harmonic metrics on a simple flat bundle, the inclusion  $\mathcal A^{s}_{Bharm}(E,D_{\xi})\subset \mathcal A^{s}_{Bflat}(E,D_{\xi})$ induces a homeomorphism $\mathcal M^{s}_{Bharm}(E,D_{\xi})\cong {\mathcal M}^{s}_{Bflat }(E,D_{\xi})$.
 
 \begin{remark}
 In case $n=1$ i.e. $\dim M=3$, $\mathcal A^{s}_{Bharm}(E,D_{\xi})$ can be seen as solutions of Hitchin's self-dual equations \cite{Hit} on a basic vector bundle and hence ${\mathcal M}^{s}_{Bharm }(E,D_{\xi}) $ can be seen as the moduli space of basic Hitchin pairs.
 In this consideration, 
 Ono shows that  the moduli space ${\mathcal M}^{s}_{Bharm}(E,D_{\xi}) $ admits a canonical hyper-K\"ahler structure \cite{Ono}.
 \end{remark}

\subsection{Moduli spaces of stable basic Higgs bundles}
Let $(M,T^{1,0}M, \eta)$ be a compact  Sasakian manifold.
Let $(E, \theta)$ be a  basic Higgs  bundle.
Denote by $\nabla^{''}$ the structure of a basic  holomorphic bundle $E$.
Then the space of basic Higgs bundle structures on a  ${\mathcal C}^{\infty}$ complex vector bundle $E$ with the  fixed basic structure $D_{\xi}=\nabla^{''}_{\xi}$ is 
\[\mathcal A_{BHiggs}(E)=\{\omega\in  A^{1}_{B}(M, {\rm End}(E))\vert D^{\prime\prime}\omega+\omega\wedge \omega=0\}\]
in an affine space modeled on $A^{1}_{B}(M, {\rm End}(E))$  via $\omega\mapsto (\nabla^{''}_{\omega}=\nabla^{''}+\omega^{0,1},\theta_{\omega}=\theta+\omega^{1,0})$ with the gauge group $GL_{B}(E)$ where $D^{\prime\prime}$ is the  operator on $A^{\ast}_{B}(M, {\rm End}(E))$ induced by $\bar\partial_{E}+\theta$ .
For $\omega \in \mathcal A_{BHiggs}(E)$, denote by  $(E_{\omega},\theta_{\omega})$ the corresponding basic Higgs bundle.

We assume that  the basic Higgs  bundle $(E,\theta)$  is stable and satisfies
$$c_{1,B}(E)\,=\,0\qquad {\rm and} \qquad \int_{M} c_{2,B}(E)\wedge(d\eta)^{n-2}\wedge \eta\,=\,0.$$
Then by Theorem \ref{harmonicmetric}, $(E,\theta)$ comes from a simple flat bundle $(E,D)$ equipped with a harmonic metric $h$.
Additionally we  assume that  the determinant of the Higgs bundle  $(E,\theta)$ is trivial.
Then the determinant of the flat bundle  $(E,D)$ is trivial.
Consider the space of stable basic Higgs bundle structures on $E$ with fixed basic structure $D_{\xi}$ as
\[\mathcal A^{st}_{BHiggs}(E, D_{\xi})=\{\omega\in \mathcal A_{BHiggs}(E)\cap  A^{1}_{B}(M, {\rm End}_{0}(E))\vert  (E_{\omega},\theta_{\omega}) \,\,  stable \}\]  with the action of $SL_{B}(E)$.
Define the moduli space
 \[{\mathcal M}^{st}_{BHiggs }(E, D_{\xi})=\mathcal A^{st}_{BHiggs}(E, D_{\xi})/ (SL_{B}(E)/\langle \zeta_{r} I\rangle).\]

Associated with  the decomposition $A^{\ast}_{B}(M, {\rm End}_{0}(E))=A^{\ast}_{B}({\frak u}_{0}(E))\oplus A_{B}^{\ast}(P_{0}(E))$, we consider the map linear map \[{\mathcal I}:A^{1}_{B}({\frak u}_{0}(E))\oplus A^{1}_{B}(P_{0}(E))\ni \alpha+\beta \mapsto \alpha^{0,1}+ \beta^{1,0} \in A^{0,1}_{B}(M, {\rm End}_{0}(E))\oplus A^{1,0}_{B}(M, {\rm End}_{0}(E))\] with the inverse
$\alpha^{0,1}-\overline{\alpha^{0,1}}_{h}+ \beta^{1,0}+ \overline{\beta^{1,0}}_{h}$ where $\overline{\alpha^{0,1}}_{h}, \overline{\beta^{1,0}}_{h}$ are adjoints with respect to $h$.
The non-abelian   correspondence constructed in the last section  says that the restriction of  ${\mathcal I}$ on $\mathcal A^{s}_{Bharm}(E)$ has the image $\mathcal A^{st}_{BHiggs}(E, D_{\xi})$ 
Consider the induced map $\overline{\mathcal I}: \mathcal M^{s}_{Bharm}(E, D_{\xi})\to {\mathcal M}^{st}_{BHiggs }(E, D_{\xi})$.
Define 
\[\mathcal A_{BHiggsharm}(E, D_{\xi})=\{\omega\in\mathcal A^{st}_{BHiggs}(E, D_{\xi}) \vert (D^{h}_{\omega})^2=0\} \]
where $D^{h}_{\omega}$ is the canonical connection of the Higgs bundle $(E_{\omega}, \theta_{\omega})$ with respect to the basic metric $h$.
By Theorem \ref{harmonicmetric},  the restriction of the inverse of  ${\mathcal I}$ on $\mathcal A_{BHiggsharm}(E, D_{\xi})$ induces the inverse of $\overline{\mathcal I}: \mathcal M^{s}_{Bharm}(E, D_{\xi})\to {\mathcal M}^{st}_{BHiggs }(E, D_{\xi})$.
We have shown the following statement.
\begin{theorem}
The map  $\overline{\mathcal I}: \mathcal M^{s}_{Bharm}(E, D_{\xi})\to {\mathcal M}^{st}_{BHiggs }(E, D_{\xi})$ is a homeomorphism.
\end{theorem}

Thus the non-abelian Hodge correspondence on compact Sasakian manifolds  can be stated as the following
\begin{theorem}
There exists a canonical homeomorphism  ${\mathcal M}^{s}_{Bflat }(E,D_{\xi})\cong  {\mathcal M}^{st}_{BHiggs }(E, D_{\xi})$.

\end{theorem}

\section{Compactness and properness}
We consider  an analogue  of Simpson's compactness \cite[Lemma 2.8]{Si2} for Sasakian case.
Simpson's compactness is an important  property for extending the non-abelian Hodge correspondence to moduli spaces of  semi-stable objects (see \cite{Si3,Si4}).

Let $(M, T^{1,0}M,\eta)$ be a $(2n+1)$-dimensional  compact  Sasakian manifold and $(E, D)$ be a flat complex vector bundle over $M$.
Assume that $(E, D)$ is semi-simple and take  a harmonic metric $h$ on $E$.
Consider   the space of (not necessarily simple) harmonic bundle structures  on $E$ as
\[\mathcal A_{harm}(E)=\{\omega \in A_{flat}(E)\vert  \nabla_{\omega} ^{\ast}\phi_{\omega}=0\}
\]
and 
the space of (not necessarily simple) harmonic bundle structures  on $E$ with fixed basic structure $D_{\xi}$ as \[\mathcal A_{Bharm}(E, D_{\xi})=\{\omega \in \mathcal A_{Bflat}(E, D_{\xi})\vert \overline{\partial}_{\omega}\overline{\partial}_{\omega}= \overline{\partial}_{\omega} \theta_{\omega}=\theta_{\omega}\wedge \theta_{\omega}=0\}.\]

We first see that  similar arguments as \cite[Lemma 2.8]{Si2} works in the weaker sense (see the remark after the proof).
\begin{proposition}\label{com}
Let  $\{D_{i}\}_{i}$ be a sequence in $\mathcal A_{Bharm}(E)$ and $(E_{i}, \theta_{i}) $ the sequence of basic Higgs bundles induced by $\{D_{i}\}_{i}$.
Assume that the coefficients of the characteristic polynomial of $\theta_{i}$  are bounded in $C^{0}$-norm.
Then there are $\tilde{D}_{\infty}\in \mathcal A_{harm}(E)$,  a subsequence of $\{D_{i}\}_{i}$ and a sequence $\{g_{i}\}_{i}$ in $  U(E)$ such that $\{g^{-1}_{i}D_{i}g_{i}\}_{i}$ converges to $\tilde{D}_{\infty}$ with respect to $L^{p}$ norm.
\end{proposition}
\begin{proof}
Considering foliation charts of the  transverse holomorphic foliation associated with $(\xi, T^{1,0})$, by Theorem \ref{t4.2},
harmonic bundles over a $(2n+1)$-dimensional  compact Sasakian manifold $M$ can be seen locally as a harmonic bundle over a complex $n$-dimensional poly-disc around each point.
The proof of \cite[Lemma 2.7]{Si2} is given by an  estimate on  a harmonic bundle  over a poly-disc and the compactness.
Hence,  without any change the same statement holds for  the basic  Higgs bundle $(E,\theta)$ induced by $D\in \mathcal A_{Bharm}(E, D_{\xi})$. 
And so we can say that $\vert \theta_{i}\vert$ are bounded.
Thus, for the decomposition $D_{i}=\nabla_{i} +\phi_{i}$,   the curvature of the unitary  connection  $\nabla_{i} $ is  bounded by the equation $\nabla_{i}^2=-\theta_{i}\wedge \overline{\theta}_{i}-\overline{\theta}_{i}\wedge\theta_{i}$.
By the Uhlenbeck compactness theorem \cite{Uh}, there exists a subsequence of $\{D_{i}\}_{i}$,  a sequence $\{g_{i}\}_{i}$ in $  U(E)$ and  a unitary connection $\nabla_{\infty}$ such that $g^{-1}_{i}\nabla_{i}g_{i}$ converges to $\nabla_{\infty}$  weakly in $L^{p}_{1}$.

Consider $\tilde{D}_{i}=g^{-1}_{i}\nabla_{i}g_{i}+g^{-1}_{i}\phi_{i}g_{i}$.
By the similar  arguments in \cite[Page 373]{Cor},  $\tilde{D}_{i}$ is bounded in $L^{p}_{1}$ and hence some subsequence weakly  converges to a connection $\tilde{D}_{\infty}$ and  $\tilde{D}_{\infty}$ is flat.
By the assumption $D_{i}\in \mathcal A_{Bharm}(E, D_{\xi})$,  $\tilde{D}_{\infty}$ is harmonic.
By the elliptic regularity $\tilde{D}_{\infty}$ is smooth.
We can say that a subsequence of $\{\tilde{D}_{i}\}_{i}$ converges to $\tilde{D}_{\infty}$ with respect to $L^{p}$ norm.

\end{proof}

\begin{remark}\label{remcom}
The Uhlenbeck compactness theorem is not compatible with fixed basic structures $D_{\xi}$.
 (cf. \cite[Example 6.4]{BH}). 
 We can not assume that   $g_{i}\in U_{B}(E)$   and $\tilde{D}_{\infty} \in  \mathcal A_{Bharm}(E, D_{\xi})$.
We can not say the convergence of basic  Higgs bundles unlike Simpson's compactness \cite[Lemma 2.8]{Si2} on compact K\"ahler manifolds.
\end{remark}
This thing is  a  difficulty  giving  the same arguments in the proof of  \cite[Lemma 2.8]{Si2}.
By Corollary \ref{finiss}, we can state:
\begin{corollary}\label{Bcom}
The basic structure induced by $\tilde{D}_{\infty}$ is isomorphic to $D_{\xi}$.
\end{corollary}
\begin{proof}
We consider the space ${\mathcal B}(E,h)$ of unitary  basic structures on $E$ as an affine space modeled on $A^{0}(M, {\frak u}(E))$ with the unitary gauge group $U(E)$ such that the action defined by
$g^{-1}\alpha g+g^{-1}D_{\xi} g$ for $g\in U(E)$, $\alpha\in A^{0}(M, {\frak u}(E))$.
By  the same proof as \cite[Proposition7.1.14]{Ko} by  considering horizontal lifts of the flow of $\xi$ associated with basic structures (see \cite[Page 16]{BK2})),  we can say that the action of $U(E)$ on ${\mathcal B}(E,h)$ is proper.
Thus the quotient ${\mathcal B}(E,h)/U(E)$ is Hausdorff.
Define the continuous map $\mu : \mathcal A_{harm}(E)/ U(E)\to  {\mathcal B}(E,h)/U(E)$ induced by 
 the equivariant continuous map $\mathcal A_{harm}(E)\ni D\mapsto D_{\xi} \in  {\mathcal B}(E,h)$ which is well-defined by the basicness of harmonic metrics (see Theorem \ref{t4.2}).
Since $\mathcal A_{harm}(E)$ consists of semi-simple flat connections  by  the Corlette theorem \cite{Cor}, 
 by Corollary \ref{finiss}, the image of $\mu$ is finite and discrete by the Hausdorffness.
 Thus, the statement follows from the assumption $D_{i}\in \mathcal A_{Bharm}(E, D_{\xi})$.

\end{proof}

We give a Sasakian analogue  of Hitchin's properness \cite[Theorem 8.1]{Hit}.
Consider  the map $\sigma: {\mathcal M}^{st}_{BHiggs }(E, D_{\xi})\to \bigoplus ^{{\rm rank} E} _{i=2}H^{0,0}_{B}(M, Sym^{i} T^{1,0\ast})$ defined by   the coefficients of the characteristic polynomial of Higgs fields $\theta$.

\begin{corollary}\label{prop}
If every semi-simple flat bundle  with fixed basic structure $D_{\xi}$ is simple, then the map $\sigma: {\mathcal M}^{st}_{BHiggs }(E, D_{\xi})\to \bigoplus ^{{\rm rank} E} _{i=2}H^{0,0}_{B}(M, Sym^{i} T^{1,0\ast})$ is proper.
\end{corollary}
\begin{proof}
Under the assumption, the limit in ${\mathcal M}^{st}_{BHiggs }(E, D_{\xi})\cong  \mathcal M^{s}_{Bharm}(E, D_{\xi})$ exists  
taking a subsequence of a sequence in the pre-image of a closed  bounded set  for the map $\sigma$ by Proposition \ref{com} and  Corollary \ref{Bcom}.
\end{proof}

\begin{example}
On the smoothly trivial complex  line bundle $M\times \C$,  for any  complex number $C$, we define the basic bundle structure by the partial derivation   $L_{\xi}+C$ (see \cite[Example 3.7.]{BK2}). 
We denote by $E_{C}$ such basic line bundle.
Then, we  show that if $C\not=0$, then any flat line bundle does not induce a basic bundle structure which is isomorphic to the basic line bundle $E_{C}$ by $[d\eta]\not=0$ in $H^{2}_{B}(M)$.
By the isomorphism between the basic cohomology $H_{B}^{1}(M)$ and the first de Rham cohomology,
a flat connection on a line bundle can be written $d+a$ for a closed $a\in A_{B}^{1}(M)$.
Suppose   this connection induces  a basic bundle structure which is isomorphic to the basic line bundle $E_{C}$.
 Then we have  a smooth $\C^{\ast}$-valued  function $b$ satisfying  $b^{-1}db(\xi)+C=0$.
 This   implies  $b^{-1}db+ C\eta\in  A_{B}^{1}(M)$ and hence $Cd\eta$ is exact.
 This happens only if $C=0$.

We assume that the universal covering of $M$ is  $\widetilde{PU}(1,1)=\widetilde{SL}_{2}(\R)$ with a left-invariant Sasakian structure.
We note that any $3$-dimensional compact Sasakian manifold $M$ with $c_{1,B}(T^{1,0})=-C[\eta]$ for $C>0$ can be deformed to a Sasakian manifold of such form (see \cite{KM}).
Consider the trivial vector bundle $E=M\times \C^{2}$ of rank $2$ equipped with the basic structure $D_{\xi}=L_{\xi}+\left(
\begin{array}{cc}
C & 0 \\
0 & -C
\end{array}
\right)$ that is $(E,D_{\xi})\cong E_{C}\oplus E_{-C}$.
We take $C$ such that $E_{C}^{2}\cong T^{1,0}$.
Then,  $\mathcal A^{s}_{Bflat}(E, D_{\xi})$ contains the flat connection defined by the Maurer-Cartan form on the Lie group $\widetilde{PU}(1,1)=\widetilde{SL}_{2}(\R)$.
Thus, in this case, $\mathcal M^{s}_{Bflat}(E, D_{\xi})$ is non-empty.
By the above argument on the line bundles, any flat bundle in  $\mathcal A_{Bflat}(E, D_{\xi})$ does not admit non-trivial flat sub-bundle.
Thus, by Corollary \ref{prop}, the map $\sigma={\rm det}: {\mathcal M}^{st}_{BHiggs }(E, D_{\xi})\to H^{0,0}_{B}(M, ( T^{1,0\ast})^2)$ is proper (cf. \cite[Theorem 8.1]{Hit}).

Define the holomorphic structure on $E_{C}$ defined by the connection $d+C\eta$.
We have the  stable Higgs bundle $E_{C}\oplus E_{-C}$ with the Higgs field $\theta=\left(
\begin{array}{cc}
0 & \lambda \\
1 & 0
\end{array}
\right)$ for any \[\lambda \in H^{0,0}_{B}(M, {\rm Hom}( E_{C}, E_{-C})\otimes  T^{1,0\ast})=H^{0,0}_{B}(M, ( T^{1,0\ast})^2).\]
Thus, $\sigma$ is surjective.

We notice that $M$ is a Seifert $S^1$-fibration over a compact hyperbolic orbifold  Riemannian surface \cite{RV}.
It would be very interesting to study relations between  moduli spaces  for  $3$-dimensional Sasakian manifolds $M$ and moduli spaces for hyperbolic (orbifold)  Riemannian surfaces studied  in \cite{Hit}, \cite{NS} and many other works.
Studies of this  thing  in greater detail will  appear in other papers.
We give some observations in  the more specific case.
Assume that $M$ is the unit tangent bundle of compact hyperbolic Riemannian surface $\Sigma$.
Then $M=\Gamma \backslash PSL_{2}(\R)=\widetilde{\Gamma} \backslash\widetilde{SL}_{2}(\R) $ where $\Gamma$ is the fundamental group of $\Sigma$ with the embedding $\Gamma \subset PSL_{2}(\R)$ associated with a hyperbolic metric on 
$\Sigma$ and $\widetilde{\Gamma}$ is a central extension of $\Gamma$ by $\Z$.
More precisely, the Sasakian structure on $M$ induces  a principal $S^{1}$-bundle over $\Sigma$ whose homotopy long exact sequence is a central extension
\begin{equation}
\xymatrix{
1\ar[r]&\pi_{1}(S^1,\, 1)=\Z\ar[r]&\pi_{1}(M) =\widetilde{\Gamma}\ar[r]&
\pi_{1}(\Sigma)=\Gamma\ar[r]&1.}
\end{equation}
In this case, consider the pull-back $\Gamma_{2}\subset SL_{2}(\R)$ of $\Gamma\subset PSL_{2}(\R)$ for  the surjection $SL_{2}(\R)\to PSL_{2}(\R)$ associated with the M\"obius transformation. 
Then the map $\widetilde{\Gamma}\to \Gamma_{2}$  which is the restriction of  the covering map $\widetilde{SL}_{2}(\R) \to SL_{2}(\R)$ is the monodromy representation of the flat connection $D$.
By $-I\in  \Gamma_{2}$, the image of  the central subgroup $\Z$  in $\widetilde{\Gamma}$  for  this representation is $\{\pm I\}$.
Identifying   $R^{s}(\pi_{1}M, SL_{2}(\C))/SL_{2}(\C)$ with ${\mathcal M}^{s}_{flat }(SL_2)$,  we regard ${\mathcal M}^{s}_{Bflat }(E, D_{\xi})$ as a subspace in $R^{s}(\pi_{1}M, SL_{2}(\C))/SL_{2}(\C)$.
Consider the subspace  $R^{s, odd}(\pi_{1}M, SL_{2}(\C))/SL_{r}(\C)$ in $R^{s}(\pi_{1}M, SL_{r}(\C))/SL_{r}(\C)$ such that 
$R^{s, odd}(\pi_{1}M, SL_{2}(\C))$ is the subspace of simple representations whose images of   the central subgroup $\Z$  in $\widetilde{\Gamma}$ are $\{\pm I\}$.
It is known that $R^{s, odd}(\pi_{1}M, SL_{2}(\C))$ is connected \cite[Theorem 9.20]{Hit}.
Since $R^{s, odd}(\pi_{1}M, SL_{2}(\C))$ contain the monodromy representation of the flat connection $D$ and  ${\mathcal M}^{s}_{Bflat }(E, D_{\xi})$ is an  open and closed set in $R^{s}(\pi_{1}M, SL_{2}(\C))/SL_{2}(\C)={\mathcal M}^{s}_{flat }(SL_2)$ by Corollary \ref{fidi}, we have  $R^{s, odd}(\pi_{1}M, SL_{2}(\C))/SL_{r}(\C)\subset {\mathcal M}^{s}_{Bflat }(E, D_{\xi})$.
Notice that the restriction of the monodromy representation on  the central subgroup $\Z$  in $\widetilde{\Gamma}$  is given by horizontal lifts of the flow of $\xi$.
Thus, the conjugacy class of the image of   the central subgroup $\Z$  in $\widetilde{\Gamma}$ is an invariant of isomorphism classes of basic bundles.
This implies \[R^{s, odd}(\pi_{1}M, SL_{2}(\C))/SL_{r}(\C)= {\mathcal M}^{s}_{Bflat }(E, D_{\xi}).\]
We notice that the space $R^{s, odd}(\pi_{1}M, SL_{2}(\C))/SL_{r}(\C)$ is homeomorphic to the moduli space of the stable Higgs bundles of rank $2$  with fixed determinant and   odd degree over the Riemannian surface \cite{Hit}.
This homeomorphism is derived  through the solutions of Hitchin's self-dual equations up to “$SO_{3}$”-gauge transformations.

We know that there is a lifting representation $\Gamma\to SL_{2}(\R)$ of  $\Gamma\subset PSL_{2}(\R)$ \cite{Cu}.
Via the composition $\tilde\rho:\widetilde{\Gamma}\to SL_{2}(\R)$  of $\Gamma\to SL_{2}(\R)$ and the surjection $\widetilde{\Gamma}\to \Gamma$, image of  the central subgroup $\Z$  in $\widetilde{\Gamma}$  for  this representation is trivial.
Thus,  considering a flat bundle $(\tilde{E}, \tilde{D})$ whose monodromy representation is $\tilde\rho$
we have a  non-empty stratum ${\mathcal M}^{s}_{Bflat }(\tilde{E}, \tilde{D}_{\xi})$ in ${\mathcal M}^{s}_{flat }(SL_2)$ which is different from ${\mathcal M}^{s}_{Bflat }(E, D_{\xi})$ as above.

\end{example}


\begin{thebibliography}{ZZZZZZ}

\bibitem[BH]{BH} D. Baraglia and P. Hekmati, A foliated Hitchin-Kobayashi correspondence, Adv. Math. {\bf 408} (2022), part B, Paper No. 108661, 47 pp.


\bibitem[BK1]{BK} I. Biswas and H. Kasuya, Higgs bundles and
flat connections over compact Sasakian manifolds, {\it Comm. Math. Phys.}
{\bf 385} (2021), 267--290.

\bibitem[BK2]{BK2}I. Biswas and H. Kasuya, Higgs bundles and flat connections over compact Sasakian 
manifolds, II: quasi-regular bundles, arXiv:2110.10644. {\it Annali della Scuola Normale Superiore di Pisa}
(to appear).

\bibitem[BG]{BG} C. P. Boyer and K. Galicki, {\it Sasakian geometry}, Oxford Mathematical
Monographs. Oxford University Press, Oxford, 2008. 

\bibitem[Co]{Cor} K. Corlette, Flat G-bundles with canonical metrics, \textit{ J. 
Differential Geom.} { \bf 28} (1988), 361--382.
\bibitem[Cu]{Cu} M.  Culler,  Lifting representations to covering groups.  {\it Adv. Math.} {\bf 59} (1986), 64--70.

\bibitem[EK]{EKA} A. El Kacimi-Alaoui, Op\'erateurs transversalement elliptiques sur un 
feuilletage riemannien et applications. {\it Compositio Math.} {\bf 73} (1990), no. 1, 57--106.

\bibitem[FM]{FM}
R. Friedman, J. W. Morgan,  
Smooth four-manifolds and complex surfaces. 
Ergebnisse der Mathematik und ihrer Grenzgebiete (3),{\bf  27}. Springer-Verlag, Berlin, 1994.
\bibitem[Fu]{Fu}
A. Fujiki,  Hyper-K\"ahler structure on the moduli space of flat bundles. Prospects in complex geometry (Katata and Kyoto, 1989), 1--83, {\it Lecture Notes in Math.}, {\bf 1468}, Springer, Berlin, 1991.

\bibitem[GM2]{GMi2}W. M. Goldman, J. J. Millson,The homotopy invariance of the Kuranishi space. Illinois J. Math. {\bf 34} (1990), no. 2, 337--367. 

\bibitem[Hi]{Hit} N. J. Hitchin, The self-duality equations on a Riemann surface, {\it 
Proc. London Math. Soc.} {\bf 55}(1987), 59--126.
\bibitem[JM]{JM}
D. Johnson, J. J.  Millson,  Deformation spaces associated to compact hyperbolic manifolds.  Discrete groups in geometry and analysis (New Haven, Conn., 1984), {\it Progr. Math.}, {\bf 67}, Birkhuser Boston, Boston, MA, 1987, 48--106


\bibitem[KT]{KT} F. W. Kamber and P. Tondeur
de Rham-Hodge theory for Riemannian foliations. 
\textit{ Math. Ann.} {\bf 277} (1987), no. 3, 415--431. 

\bibitem[Ka]{Ka} H. Kasuya, Almost-formality and deformations of representations of the fundamental groups of Sasakian manifolds, {\it Ann. Mat. Pura Appl.} {\bf 202} (2023), 1793--1801.

 \bibitem[KM]{KM} H. Kasuya and  N. Miyatake, Uniformizations of compact Sasakian manifolds, {\it International 
Mathematics Research Notices},  Volume 2024, Issue 10, May 2024, Pages 8313--8328

\bibitem[Kob]{Ko} S. Kobayashi, \textit{Differential geometry of complex vector bundles}, 
Princeton University Press, Princeton, NJ, Iwanami Shoten, Tokyo, 1987.
\bibitem[LT]{LT}
M. L\"ubke, A. Teleman,  The Kobayashi-Hitchin correspondence. World Scientific Publishing Co., Inc., River Edge, NJ, 1995.
\bibitem[NS]{NS} B. Nasatyr and B. Steer, Orbifold Riemann surfaces and the 
Yang-Mills-Higgs equations, {\it Ann. Scuola Norm. Sup. Pisa} {\bf 22} (1995), 595--643.
\bibitem[Ono]{Ono} T. Ono, 
Moduli Spaces of the Basic Hitchin equations on Sasakian three-folds. arXiv:2409.16625
\bibitem[Ra]{Ra} J. H. Rawnsley, Flat partial connections and holomorphic structures in 
$\mathcal C^{\infty}$ vector bundles, {\it Proc. Amer. Math. Soc.} {\bf73} (1979), 
391--397.
\bibitem[RV]{RV}
F. Raymond,  A. T. Vasquez, 
3-manifolds whose universal coverings are Lie groups. 
{\it Topology Appl.} {\bf 12} (1981), no. 2, 161--179.
\bibitem[Si1]{Si1} C. T. Simpson, Constructing variations of Hodge structure using Yang-Mills 
theory and applications to uniformization, \textit{Jour. Amer. Math. Soc.} {\bf 1} (1988),
867--918.

\bibitem[Si2]{Si2} C. T. Simpson, Higgs bundles and local systems, \textit{Inst. Hautes \'Etudes 
Sci. Publ. Math.} No. {\bf 75} (1992), 5--95.
\bibitem[Si3]{Si3}C. T. Simpson, 
Moduli of representations of the fundamental group of a smooth projective variety I, \textit{Inst. Hautes \'Etudes 
Sci. Publ. Math.}, No. {\bf 79}, (1995) 47--129.
\bibitem[Si4]{Si4}C. T. Simpson, 
Moduli of representations of the fundamental group of a smooth projective variety II, \textit{Inst. Hautes \'Etudes 
Sci. Publ. Math.}, No. {\bf 80}, (1994) 5--79.
\
\bibitem[Ta]{Tan} N. Tanaka, A Differential Geometric Study on strongly pseudoconvex
CR manifolds, Lecture Notes in Math., {\bf 9}, Kyoto University, 1975.
\bibitem[Uh]{Uh}K. Uhlenbeck, Connections with $L^{p}$ bounds on curvature,  {\it Comm. Math. Phys.} ,{\bf 83} (1982), 31--42. 

\bibitem[UY]{UY} K. Uhlenbeck and S.-T. Yau, On the existence of Hermitian-Yang-Mills 
connections in stable vector bundles, {\it Comm. Pure Appl. Math.} {\bf 39} (1986), 
257--293.

\bibitem[Wa]{Wa} A. W. Wadsley, Geodesic foliations by circles, {\it J. Differential 
Geometry} {\bf 10} (1975), 541--549.

\bibitem[We]{Web} S. Webster, Pseudo-Hermitian structures on a real hypersurface, 
\textit{J. Differential Geom.} {\bf 13} (1978), 25--41.


\end{thebibliography}
\end{document}